\documentclass [12pt,a4paper,reqno]{amsart}
  \input amssymb.sty

\usepackage{epsfig}

\newcommand{\ef}{\end{equation}}
\chardef\bslash=`\\ 

\newtheorem*{thm*}{Theorem}

\newtheorem{lem}{Lemma}[section]

\newtheorem{prop}{Proposition}[section]
\newtheorem{prop*}{Proposition}

\theoremstyle{definition}
\newtheorem{defn}{Definition}[section]
\newtheorem{examp}{Example}
\newtheorem*{examp*}{Example}
\newtheorem*{remark*}{Remark}
\newtheorem*{CC*}{Crossover Conjecture}
\newtheorem*{Note*}{Note}
\newtheorem*{defn*}{Definition}

 \theoremstyle{remark}
\newtheorem{remark}{Remark}[section]

\newcommand{\wt}{\widetilde}

 \renewcommand{\sectionmark}[1]{}

\newcommand{\ve}{\varepsilon}

\newcommand{\g}{\gamma}

\newcommand{\lm}{\lambda}

 \date{}
\begin{document}

\title  [Short strings] { $\frak {sl}_2$-actions along short strings for spin blocks}
\author[O. Alon, M. Schaps, M. White]
{ Ortal Alon, Mary Schaps, and Michal White}
 \address{Department of Mathematics\\
 Bar-Ilan University, 52900 Ramat Gan, Israel}
 \email {Ortal.alon@gmail.com}

 \address{Department of Mathematics\\
 Bar-Ilan University, 52900 Ramat Gan, Israel}
\email{mschaps@macs.biu.ac.il}

 \address{Department of Mathematics\\
 Bar-Ilan University, 52900 Ramat Gan, Israel}
\email{avner18@bezeqint.net}


 \begin{abstract} 
 The problem of source algebra equivalences between blocks at the ends of the maximal strings of spin blocks of the symmetric and alternating groups has recently been settled \cite{LS}, but so far there has not even been a candidate of the tilting complex defining the reflections of the internal blocks of the string.
 
Our aim in this paper is to propose a definition for the mappings $E_i$ and $F_i$ and for their  divided powers.  The solution we propose here is to halve the operators only when both halves are isomorphic, which, on the level of the Grothendieck group, corresponds to working with paired simples as pairs, and crossing over between the symmetric and alternating groups, in order to send paired simples to paired simples.
 This is equivalent to working with the irreducible supermodules, as in \cite{BK1}, though we don't introduce the supermodule language into this paper. The main difference between our approach and theirs is the crossovers, which occasionally force extra halving of modules.

 The definition is based on the method of permutation modules from \cite{KS}. We apply it to short strings and show that, at least on the individual strings, it is compatible with the form of the tilting complex used by Chuang and Rouquier.

 \end{abstract}
\thanks{Partially supported by the Bar-Ilan Research Authority}
\thanks{Portions of this work appeared in the Master's thesis of the first and last author at Bar-Ilan University}

\maketitle

 \section{Introduction}\label{introduction}

Let $G$ be a finite group and let $k$ be a field of characteristic $p$, where
$p$ divides $|G|.$ Assume that $k$ is sufficiently large that it is a splitting field for all relevant finite groups.
 The $p$-blocks of the symmetric
groups are determined by a partition $\rho$ called the $p$-core and by a non-negative integer $w$ called the weight. Rouquier and Chuang \cite{CR} showed that all blocks of a fixed weight in the symmetric and alternating groups are derived equivalent.  The Chuang-Rouquier method, called $\frak sl_2$-categorization,  uses Lie group methods, including reflection functors.

The symmetric and alternating groups have central extensions $\tilde S_n$ and $\tilde A_n$ with kernel $C_2$, the cyclic group of order $2$, which we will loosely refer to as covering groups (although, for a few small values of $n$, the kernel of the extension is not contained in the commutator of the group). We assume henceforward that $p$ is an odd prime, in which case all blocks can be divided into ordinary blocks of $S_n$ or $A_n$ or \textit{spin blocks}, whose characters all take the value $-1$ on the non-trivial element of the center.

As with blocks of the symmetric groups, the spin blocks of the covering groups are determined by a non-negative integer $w$ called the weight, and by a partition $\rho$ called the $p$-bar core. However, the $p$-bar cores must be strict partitions, cannot contain any parts divisible by $p$, and cannot contain parts congruent to $i$ and to $p-i$ for any $i$ satisfying $1 \leq i \leq p-1$.
Any pair $(\rho, w)$ for $w > 0$ determines exactly one block of $\tilde S_n$ and one block of $\tilde A_n$.  In the sequel, we will denote the pair $(\rho, w)$ by $\rho^w$.

The obstruction to making an immediate generalization of the Chuang-Rouquier \cite{CR} result to the spin blocks of $\tilde S_n$ lies in the fact that blocks which, by the combinatorics of the reflection functors, should seemingly be derived equivalent, do not always have the same number of simple modules.  Since the number of simples is invariant under derived equivalence, this meant that they could not in fact be in the same derived equivalence class.
The solution to this dilemma is the following, inspired by the results of \cite{KS} and formulated explicitly in \cite{AS}:

\begin{CC*} (Kessar-Schaps)  If $p$ is an odd prime, then among all the spin blocks of $k\tilde S_n$ and $k\tilde A_n$, there are exactly two derived equivalence classes for each weight $w > 0$, and for each $p$-bar core there is exactly one block of weight $w$ in each equivalence class.  The extremal points of the maximal strings in the crystal graph correspond to Morita equivalent blocks, making the appropriate crossover from $k\tilde S_n$ to $k\tilde A_n$ if the parities differ.
\end{CC*}

  Let $kG=\oplus B_i$ be a decomposition of the group algebra
into blocks, and let $D_i$ be the defect group of the block $B_i,$ of order
$p^{d_i}.$ 
If the Crossover Conjecture can be proven, one of the intended applications is to the proof of Broue's conjecture.  Brou\'e \cite{B1},\cite{B2} has conjectured that if $D_i$ is abelian and $B_i$ is a block with defect group $D_i$, then $B_i$ and $b_i$ are derived equivalent, i.e., the bounded
derived categories $D^b(B_i)$ and $D^b(b_i)$ are equivalent.

\section{The block-reduced crystal graph}

The roles of $\tilde S_n$ and $\tilde A_n$ are much more symmetrical in the spin case then they are for the ordinary representations of the symmetric and alternating groups.
Over a field of characteristic zero, the irreducible representations of both correspond to partitions, but now they are strict partitions, containing no repeated parts, not even the part $1$.  Whereas in the symmetric case there was a one-to-one correspondence of partitions and irreducible representations, in the spin case the matter is determined by the parity of the partition.

Let $n(\lm)$ denote the number of parts in the partition $\lm$, and let $|\lm|$ denote the sum of the parts. denotes the sum of the parts. Each strict partition $\lm$ of $n$ has a parity
$$\ve = \ve(\lm) \equiv |\lm| - n(\lm) \mod~2.$$  The partition is called even if $\ve = 0$ and odd if $\ve = 1$.  Over a field of characteristic $0$, an even strict partition labels two conjugate irreducible representations of $\tilde A_n$ and one of $\tilde S_n$.  An odd strict partition labels one irreducible representation of $\tilde A_n$ and two of $\tilde S_n$.

Returning temporarily to the symmetric group case, the irreducible modules for all the blocks of all the symmetric groups can be arranged into a labeled graph called the \textit{crystal graph}, with the edges connecting irreducible modules of $S_n$ to simple modules of $S_{n+1}$ labeled by the residues modulo $p$. A prescription for the edges can be found, for example, in \cite{K1}.  A \textit{maximal string} is a maximal connected sequence of simple modules joined by edges with the same label.  The reflection functors of Chuang-Rouquier reflect the maximal strings around their midpoint, with the reflection preserving the weights of the blocks to which the simples belong.  Furthermore, it was shown by Scopes \cite{Sc} that the extremal points of the maximal strings are actually Morita equivalent.  A Lie theoretic version of this result, due to Brundan and Kleshchev, can be found in section 11 of \cite{Kl}.

Recalling that $p$ is odd, we set $t = (p-1)/2$.

We now define the Scopes involutions.  These were defined by Scopes \cite{Sc} in the symmetric group case, and generalized to the spin case by Kessar in \cite{K}.  Let $D_p$ be the set of all partitions which are strict for all parts not divisible by p.
\begin{defn} For $0<i\le t$, the \textit{Scopes involution} $K_i: D_p \to D_p$ will interchange the parts congruent to $i$ and $i+1$ and also interchange the parts congruent to the complements $p-i$ and $p-i-1$.  For $i=0$, we have an involution $K_0: D_p \to D_p$ which exchanges $ap+1$ and $ap-1$, as well as adding a part $1$ if it is not present, and removing it, if it is.
\end{defn}

 \begin{remark}
The \textit{content} of a strict partition is obtained by filling in each row of the Young diagram by copies of the sequence $$0,1,\dots,t-1,t,t-1,\dots,0,0,1,\dots$$, and letting the $i$-th coordinate of the content be the number of instances of $i$.
We have a geometric realization of a block-reduced version of the crystal graph in ${\bf R^{t+1}}$ given by placing the vertex labeled by $\rho^w$ at the point $\g(\rho^w)$ given by its content.
Letting the variables be $x_0,...,x_t$, we let an edge labeled by $i$ be represented by a line of length one parallel to the $x_i$ axis.   In Fig. 1, we give the geometric realizations for $p=5$, in a three-dimensional representation. To make the three-dimensional representation easier to view, we have drawn the positive $\g_t$ axis going down rather than up.  Within each layer, the edges labeled $0$ go from the upper right to the bottom left, and the edges labeled $1$ from the upper left to the bottom right.  All blocks in a horizontal line on the two-dimensional representation, in both figures, have the same rank, i.e., represent blocks in the same $\tilde S_n$ or the same $\tilde A_n$. Each layer has central symmetry.
Note that each layer contains a translated copies of the previous layer.
\end{remark}
\begin{figure}

\centerline{\psfig{figure= 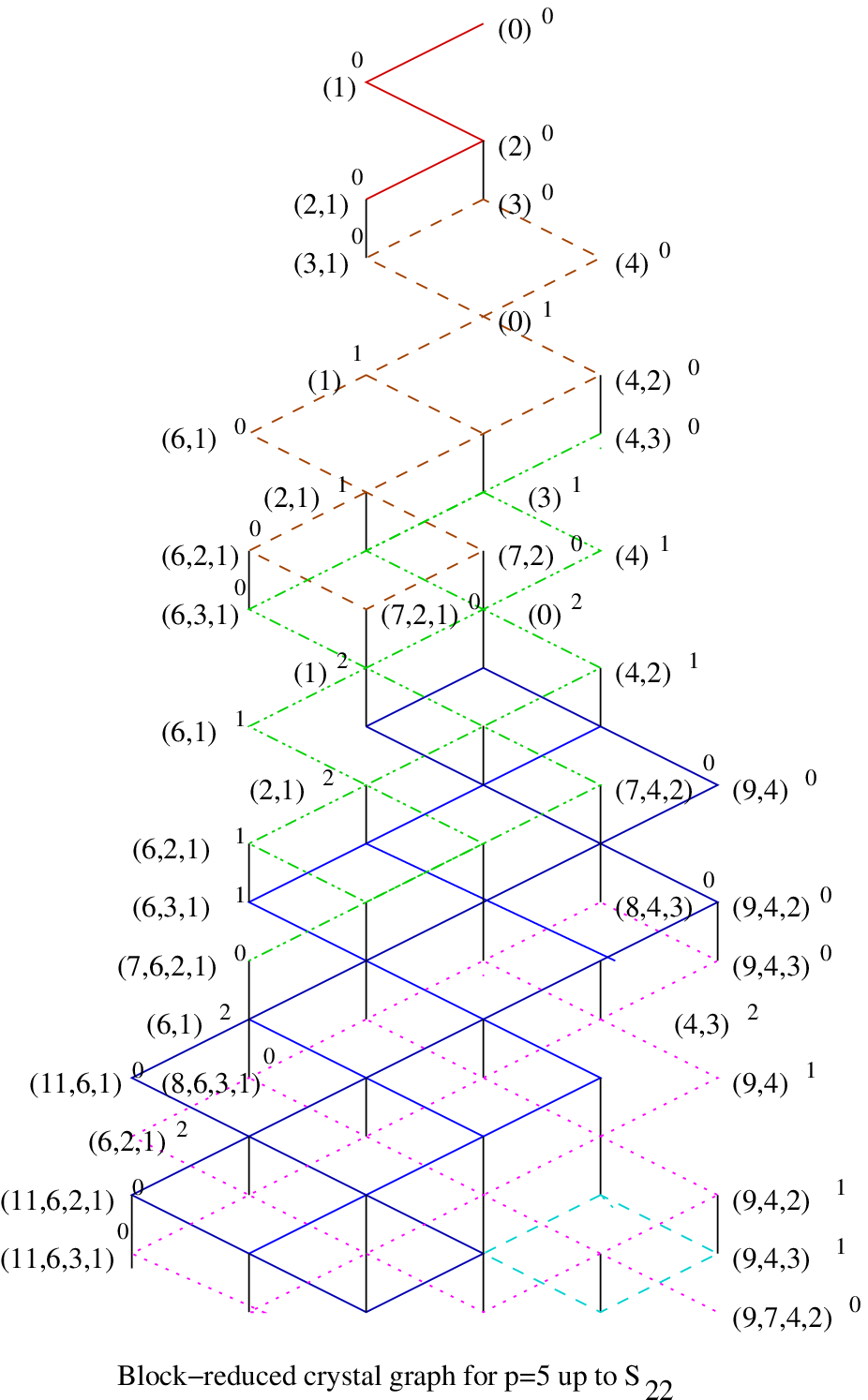}}

\caption{}
\end{figure}
\section{Divided powers}

We now propose a candidate to give ``divided powers" between symmetric blocks of a string. Let $(K, \mathcal O, k)$ be a modular triple, with $K$ of characteristic $0$ and $k$ of characteristic $p$.  The ring $\mathcal O$ is a local integral domain, with residue field $k$ and quotient field $K$.

To establish our notation, let $G = \widetilde S_n$, let $G'$ be either $\widetilde S_n$ or $\widetilde A_n$, let $G''$ be either$\widetilde S_m$ or $\widetilde A_m$, for an integer $m<n$. Let $\alpha = n-m$. Let $b$ be the block idempotent of a block $B'$ of $RG'$ with core $\nu$ and let $c$ be the block idempotent of a block $B''$ of $RG'$ with core $\mu$. Since $G'$ and $G''$ are subgroups of $G$, $\mathcal OG$ is a permutation module over $G'$ acting from the left and $G''$ acting on the right.  Let us assume that the two blocks have the same defect group $D$ and are extremal blocks of an $i$-string, with common weight $w$.  By \cite{Ca} we can presume that $D$ is the $p$-Sylow subgroup of $\widetilde S_{pw}$, and that $\wt S_{|\nu|}$ and $\wt S_{|\mu|}$ are embedded into $\wt S_n$ and $\wt S_m$ as liftings of the permutations the numbers complementary to the first $pw$ numbers on which $\wt S_pw$ acts.

Let $\sigma$ be an element of $\widetilde S_{|\nu|}-\widetilde A_{|\nu|}$ and if $w>0$, let $\tau$ be an element of $N_{\widetilde S_{pw} }(D)-\widetilde A_{pw}$.
Let us define the parity $\epsilon(G')$ to be $1$ if $G' = \widetilde S_n$ and $0$ if $G' = \widetilde A_n$, and similarly for $G''$.  We define $\bar G' = G' \bigcap \wt S_{|\nu|}$, which equals either $\wt S_{|\nu|}$ or $\wt A_{|\nu|}$ and similarly
$\bar G'' = G'' \bigcap \wt S_{|\mu|}$.
In Lemma 6.1 of \cite {KS}, the first section of the lemma is a proof of the following:

\begin{lem}  Let $G, G', G'', b, c, \nu, \mu$ be as above.  If $w>0$, then in the decomposition of
$\mathcal OGbc$ as a $\mathcal O[G' \times G''^{op}]$ bimodule, the direct sum of sub-bimodules of vertex $D$ has indecomposable components which corresponds to the decomposition of the semisimple module $k\wt S_{|\nu|} \bar b \bar c$ as a module over $k[\bar G' \times \bar G''^{op} \times R_{G'G''}],$ where
\begin{itemize}
\item If $\epsilon(G')=1, \epsilon(G'')=1$, then $R_{G'G''} = <(\tau,\tau^{-1})>$
\item If $\epsilon(G')=1, \epsilon(G'')=0$, then $R_{G'G''} = <(\tau,\tau^{-1}\sigma^{-1})>$
\item If $\epsilon(G')=0, \epsilon(G'')=1$, then $R_{G'G''} = <(\sigma\tau,\tau^{-1})>$
\item If $\epsilon(G')=0, \epsilon(G'')=0$, then $R_{G'G''} = <(\sigma\tau,\tau^{-1}\sigma^{-1})>$
\end{itemize}
\end{lem}

We now use this result to calculate the multiplicity of each of the components of $\mathcal OGbc$ and show that there are never more than two isomorphism classes.

\begin{prop}\label{res}
 Let $G, G', G'', b, c, \nu, \mu$ be as above. If $\nu^w$ and $\mu^w$ are symmetrically placed blocks in an $i$-string, then as a $\mathcal O[G' \times G''^{op}]$ module, $\mathcal OGbc$ decomposes into $d$ isomorphism classes of modules, each occurring with multiplicity $2^\eta\beta$, where $d,\eta$ and $\beta$ are given as follows:
\begin{itemize}
\item Case 1: $w>0$.
\begin{enumerate}
 \item If $i \neq 0$, then $\beta= \alpha!$ and if $i = 0$, then $\beta = \frac{\alpha!}{2^\frac{\alpha-\epsilon(\alpha)}{2}}$.
 \item $d=1$ unless $\epsilon(G')=\epsilon(\nu)$ and $\epsilon(G'')=\epsilon(\mu)$, in which case $d=2$.
\item We have $\eta=\frac{\alpha -\epsilon(\alpha)}{2} +(1 - \epsilon(G'))$.
\end{enumerate}

 \item Case 2: $w=0$. The values of $\beta$ and $\eta$ are as in the previous case, but $d= 2^{2-|\epsilon(G')-\epsilon(\nu)|-|\epsilon(G'')-\epsilon(\mu)|}$
\end{itemize}
\end{prop}
 \begin{proof}
 \begin{enumerate}

 \item The values of $\beta$ are given in \cite{K}, and correspond to the number of different ways to get from $\nu$ to $\mu$ by removing $i$-nodes.  The formula printed in \cite{K} contains an error, but that has been corrected here.

     \item From the arguments in
 \cite{KS}, it is clear that $k \bar G \bar b \bar c$ decomposes as a $k[\bar G \times \bar G''^{op}]$-module into $1$, $2$, or $4$ different isomorphism classes of tensor products of a simple modules of $k \bar G' \bar b$ and the dual of a simple module of $k \bar G'' \bar c$, all with the same multiplicity $2^\eta$. The exact number $d$ is given by the formula in Case $2$, since a core produces two blocks precisely when the group and the core have the same parity.
 There are two possibilities for the first exactly when $\epsilon(G')=\epsilon(\nu)$ and there are two possibilities for the second exactly when $\epsilon(G'')=\epsilon(\mu)$.

   If $w=0$, then $d$ is the number of isomorphism classes in the bimodule.  If $w>0$, the action of $R_{G'G''}$ permutes each pair where there is a pair, so in the case where there are one or two isomorphism classes, there is only one in the set of fixed points.  When there are four isomorphism classes, they are permuted in pairs, leaving two different classes in the end.

     \item It remains to calculate $2^\eta.$  If $\epsilon(G')=1$, then this is the restriction bimodule, and by the branching rules given in \cite{St}, we get $\eta = \frac{\alpha-\epsilon(\alpha)}{2}$. When $\epsilon(G')=0$, so that we have an extra index of $2$, we must add $1$ to get $\eta = \frac{\alpha+\epsilon(\alpha)}{2}$. The formula in the lemma combines the two cases into one.
\end{enumerate}

\end{proof}

\begin{defn} We denote the representatives of the indecomposable bimodules from which $\mathcal OGbc$ is composed by $E^{(\alpha)}_1, \dots, E^{(\alpha)}_d$.
We then set $E^{(\alpha)}=\bigoplus_{j=1}^d E^{(\alpha)}_j$, and call it the \textit{restriction bimodule}. It depends on the parities of $G'$ and $G''$.
\end{defn}

So far we have defined the restriction bimodule only between equal weight blocks, symmetrically placed in a string.  If the blocks are at the ends of a string, then they will can be used to give Morita equivalences, as in \cite{LS}.  We want to represent them as divided powers of tensor products of bimodule which restrict one block at a time down the string.

\section{Short strings}

We now consider the consequences of this definition of divided powers in small cases. The length of a string is the number of edges.
By a \textit{short string} we mean an i-string of length less than or equal to three.
These are the strings for which we will try to understand the divided powers between blocks which are not symmetrically placed on their string.

 All short strings are translations of short strings with weight $0$ at each end, since any short string can be translated up until it reaches this situation.  We will call these strings \textit{fundamental}, and the fundamental short strings of types
 $$(0),$$ $$(0, 0)$$ will be called \textit{trivial} because they have no internal blocks.  We now
 give a complete description of the non-trivial fundamental short strings.  Within the proof we will also describe the string of lowest rank satisfying the conditions, and these strings can all be found in Figure 1 by substituting $p=5$ in the general formula.

We use a condensed notation $(n_0,\dots,n_t)$ for the cores, where $|n_i|$ is the number of parts congruent to $i$ or $p-i$, and it is positive when the parts are congruent to $i$, negative when they are congruent to $p-i$.  The source of this notation is in the abelian normal subgroup of the Weyl group of the relevant affine Lie algebra.

\begin{lem}\label{short} The structure of the non-trivial fundamental short $i$-strings is independent of $p$,  and depends on $i$ as follows:
\begin{enumerate}
\item If $i = 0$, then the only non-trivial fundamental short string has weight type $$(0, 1, 1, 0),$$ and the cores of the middle terms of the short string satisfy $n_1=0$ or $n_1=1$.
\item If $i=t$, then the non-trivial fundamental short strings string have weight type $$(0,2,0),$$ $$(0,4,4,0),$$ and the cores of the middle terms satisfy $|n_t|= 0$ or  $|n_t|=1$ respectively.
\item If $0 < i < t$, then the non-trivial fundamental short strings have weight type $$(0,1,0),$$ $$(0, 2, 2, 0),$$ and the cores of the middle terms satisfy $|n_i-n_{i+1}|=0$ or $|n_i-n_{i+1}|=1$, respectively.
\end{enumerate}
\end{lem}

\begin{proof}
The middle terms of the short string are translations of trivial fundamental short strings.   A short $i$-string of type $(0)$ must arise from a core which is mapped to itself under $K_i$, and this cannot occur for $i=0$, since $n_1$ is an integer and is mapped to $1-n_1$ under the Scopes involution $K_0$. For $0<i<t$, this happens when $|n_i-n_{i+1}|=0$, and for $i=t$ for $n_t=0$.

 A short $i$-string of type $(0,0)$ must arise from two cores which are mapped into each other by $K_i$, since they are extremal in their $i$-string, and which differ in rank by one, since they are adjacent.

   Thus in each case we first show that the parts affected by $K_i$ in the trivial fundamental short strings are as described in the lemma, and then deduce the weight from  \cite{AS}.

\begin{enumerate}
\item If $i = 0$, the difference in ranks in a trivial string of type $(0,0)$ depends only on $n_1$ and is equal to the length of the string, $d=|n_1-(1-n_1)|=|2n_1-1|$.  For every value of $n_1$ except $0,1$, we have $d\geq 3$. By \cite{AS}, the weight $w$ at which these terms first become internal is $1$.

\item If $i=t$, then a non-trivial fundamental string of length $d$ satisfies $d=|n_t|+2$ and the weight of the internal block next to the ends is $w=2|n_t|+2$.  We get $d=2$ exactly when $|n_t|=0$, and then $w=2$, giving weight types $(0,2,0)$.  We get $d=3$ when $n_t=1$, and then $w=2+2=4$, giving weight type $(0,4,4,0)$.  

\item If $0 < i < t$, then $d=|n_i-n_{i+1}|+2$ and $w=|n_i-n_{i+1}|+1$.  The only possible values of $d$ giving a non-trivial fundamental $i$-string are $d= 2,3$, in which cases $w = 1,2$ and $|n_i-n_{i+1}|=0,1$, which gives the weight types
    $$(0,1,0),$$ $$(0, 2, 2, 0),$$ as required. Any core with $n_i=n_{i+1}$ produces a block of weight $1$ in a string of type $(0,1,0)$. Any core with $n_i= n_{i+1}\pm 1$ produces a block of weight $2$ in a string of type $(0,2,2,0)$.
\end{enumerate}

\end{proof}

Our aim is this section is to analyze the restriction maps between modules which occur in short strings of blocks of defect $1$ and $2$.  We first categorize all strings in which blocks of weight $1$ or $2$ can be internal, since external blocks are already taken care of in \cite{LS}.

\begin{lem}\label{max}  Spin blocks of weight $w$ can occur as internal blocks only in strings of length less than or equal to $2w+1$.
\end{lem}
\begin{proof}
 This is an immediate result of Corollary 5.2.1 from \cite{AS}:
The maximal strings in the block-reduced crystal graph are symmetrical, with the weights increasing toward the center and the successive differences decreasing.

If the length of the string is even, then the maximal possible sequence of weight differences would be $$w,w-1,\dots,1,1,\dots,w-1,w,$$ and if it is odd, then the maximal possible sequence would be $$w,w-1,\dots,1,0,1,\dots,w-1,w.$$  In the first case the maximal possible length is $2w$, and in the second, it is $2w+1$.

\end{proof}

\begin{lem}\label{1 and 2}
For $w = 1$,  the only possible weights sequences in strings having $1$ as an internal weight are
\begin{itemize}
\item $(0,1,1,0)$ for $i=0$, and
\item $(0,1,0)$ for $0<i<t$.

\end{itemize}

For $w=2$, the only possible weight sequences in which a block of weight 2 occurs as an internal blocks are
\begin{itemize}
\item $(0,2,3,3,2,0)$ or $(1,2,2,1)$ for $i=0$,
\item $(0,2,0)$  for $i=t$, and
\item $(1,2,1)$,$(0,2,2,0)$ or $(0,2,3,2,0)$ for $0<i<t$,

\end{itemize}
\end{lem}
\begin{proof}  For $w=1$, we have $d \leq 3$ by Lemma \ref{max} and for fundamental short strings these cases have already been handled in Lemma \ref{short}. Any non fundamental short strings would have the block of weight $1$ already external. For $w=2$, the fundamental short strings were calculated in Lemma \ref{short}, and the only possible non-fundamental string would have to have an internal vertex of weight $2$ and and extremal vertex of weight $1$, which can occur only in the translated weight types $$(1,2,1),$$ $$(1,2,2,1),$$ each with its proper $i$. It remains to consider $d=4,5$ by Lemma \ref{max}, and we also know, by Corollary 5.2.1 in \cite{AS}, that the possible sequences of weight differences in the two cases are $2,1,1,2$ and $2,1,0,1,2$.  These give the desired weight types, which are $(0,2,3,2,0)$ and $(0,2,3,3,2,0)$, respectively. It is possible to find examples of all the string types mentioned in the lemma in the block-reduced crystal graph in Figure 1.
\end{proof}

We now consider the restriction functors which will be needed for the categorification.
These are supposed to be liftings of restriction operators on the Grothendieck groups of the different blocks in the string.

We want them to be exact functors, defined by tensor products with bimodules over $\mathcal O$, and we want the composition of all the mappings from one end of the string to the other to give us a predictable number of copies of the module we defined above as $E^{(\alpha)}$. In so far as we succeed in defining the restriction bimodules, they  will be labeled as $E_{i,n}$, where $i$ is the type of string and $n$ is the rank of the group $G'=\widetilde S_n$ or $G'=\widetilde A_n$ on which the operator is acting.
 Each is supposed to go from one block to the next block in its $i$-string, a block of a group $G''$ of rank $n-1$ and not to contain multiple copies of bimodules from the same isomorphism class. Finally, we want the $E_{i,n}$ to preserve the type of the block.  It appears that in almost all cases the ordinary restriction bimodule combined with cutting to the correct block will give the desired map.

Every block of defect $1$ has either $t$ or $2t$ simple modules, depending on parity.  If the parity is odd, then the block of $\widetilde A_n$ will have $t$ simples occurring, a block of type $M$ and the corresponding block of $\widetilde S_n$ will have $2t$ simples, a block of type $Q$\cite{BK1}.  If the parity of the simples is even the two groups will be reversed.  We will  begin with the blocks of type $M$, which are simpler.

We propose the following definition, which we will then check for various short strings and various blocks of defect $1$ and $2$:

\begin{defn}  The bimodule $E_{(i,n)}$ is given up to isomorphism by taking the direct sum of one indecomposable bimodule from each isomorphism class of $\mathcal O \wt S_nb_nc_{n-1}$ as an $\mathcal O [G' \times G''^{op}]$ bimodule, where the parities of $G'$ and $G''$ are so chosen that each block in the $i$-string is carried to a block of the same type. Here $b_n$ and $c_{n-1}$ are the block idempotents of two adjacent blocks of the same type on the string.  There are actually two different bimodules, $E_{(i,n)}^M$ and $E_{(i,n)}^Q$, one for the block of type $M$ and one for the block of type $Q$, but we will not distinguish in the notation.
\end{defn}

By results of Juergen Mueller \cite{M}, the Brauer trees of blocks of weight $1$ and type $M$ are all linear. The exceptional divisor has multiplicity $2$, and its position on the tree depends on the number $s$ of $n_i$ in the root $t$-tuple which are different from zero.
Furthermore, the exceptional vertex corresponds to the unique irreducible in which the prime $p$ occurs as a part.

\begin{examp} We let $p=5$, and take the smallest block with defect $1$, that of the group $A_5$, which happens to be of type M.  It lies in a string in the block-reduced crystal graph of length $3$ for $i=0$, containing the blocks $(6,1)^0$,$(1)^1$, $\emptyset^1$,  and $(4)^0$. We now look at the string in the crystal graph, i.e., a string of simple modules, which we represent by their labeling partitions and separate by arrows.  Since were are trying to lift restriction, we write the string of blocks going from blocks of highest rank groups to lower rank groups. The corresponding string of simple modules in the crystal graph is:
$$ (6,1) \rightarrow (5,1) \rightarrow (4,1) \rightarrow (4).$$
Other strings of simple modules of length 3 are
$$ (6,2,1) \rightarrow (5,2,1) \rightarrow (4,2,1) \rightarrow (4,2)).$$
$$ (6,3,1) \rightarrow (5,3,1) \rightarrow (4,3,1) \rightarrow (4,3)).$$
$$ (8,6,3,1) \rightarrow (8,5,3,1) \rightarrow (8,4,3,1) \rightarrow (8,4,3)).$$
\end{examp}

For blocks of defect $1$, a block with the same top and socle must be either a projective module or a uniserial submodule, and the two cases can be distinguished by the dimensions.  As a check of our calculations, we know from Lemma \ref{res} that the total restriction from one end of the string to the other should be $3!$ times the lowest rank simple module, since here $\beta = \frac{3!}{2^\frac{3-1}{2}}=3$, $d=1$ and $\eta = 2^{\frac{3-1}{2}}$.  We illustrate this is the first of the four cases above:

\begin{examp}  Again assume $p=5$.  We apply the restriction functor $E_{0,7}$ to the defect $0$ module $L_{(6,1)}$ in $\widetilde A_7$ of dimension $20$, labeled by the strict partition $(6,1)$.  The resulting submodule of $\widetilde A_6$ will be indecomposable.  The Brauer graph of the principal module of $A_6$ is a star with exceptional vertex corresponding to the irreducible labeled by $(5,1)$.  We expect $E_{0,7}$ of $L_{(6,1)}$ to be an indecomposable projective module beginning and ending with $L_{(5,1)}$.  Since both the irreducible modules $L_{(5,1)}$ and $L_{(3,2,1)}$ have dimension $4$, this must be the projective module $P$ with $L_{(5,1)}$ as its head.

We now act on $P$ by the next restriction in the string, $E_{0,6}$. The restriction of the four dimensional simple module labeled by $(5,1)$ is a module whose head and socle are the two dimensional module labeled by $(4,1)$, so that in fact it must be unique indecomposable submodule of the projective module $P'$ of dimension 10 with top and socle given by this simple.  Thus we have
$$E^2(S) = (P')^2.$$
We thus can define the divided power $$E^{(2)} = P'.$$  Finally, we restrict to $A_4$ by $E_{0,5}$.  Each of the three copies of the simple labeled by $(4,1)$ restricts to the simple of the same dimension labeled by $(4)$, while the simple labeled by $(3,2)$ falls into another block, not in the same string.  Thus altogether
$$E^3(L_{(6,1)}) = L_{(4)}^6;E^{(3)}(L_{(6,1)}) = L_{(4)}.$$ The divided powers of the functor $F$ are precisely in the opposite direction, i.e., $F(L_{(4)}) = P'$, $F^{(2)} = P$, and $F^{(3)}= L_{(6,1)}$.

Now let us turn to the blocks of type $Q$.  There are two modules in $\widetilde S_7$ with the label $(6,1)$, which we will denote by $L_{(6,1)}^+$ and $L_{(6,1)}^-$.  Each has degree $20$.  The result of restricting one of them to $\widetilde S_6$ is the irreducible projective of dimension $20$ whose top and socle determine one of the two modules with label $(5,1)$, and the two can be distinguished by passing to the quotient field $K$ of $R$.  However, from that point on, the restrictions can no longer be distinguished by passing to $K$, because the restriction of each of the two associate simples has the same composition factors over $K$, these being one copy of each of the two associate simples with label $(4,1)$.  Over $R$ the top and socle are interchanged when one passes from one associate to the other, but over $K$ this difference is lost.

Brundan and Kleshchev \cite{BK1} have proven that when one looks at the modules involved as supermodules, the restriction of a simple module must be an indecomposable module with the same irreducible top and socle.  The irreducible top is, in fact, the next module in the string in the crystal graph.  If one looks at ordinary modules, this identification of top and socle ceases to be true. In the example with which we are dealing, the restriction of either of the modules with label $(5,1)$ is a module with
top equal to one of the two associates labeled by $(4,1)$, and socle labeled by the other.  In the restriction of the other, this will be reversed. One can decide, in trying to label modules by $+$ of $-$, that one will always label as $+$ the one at the top of the restriction of the $+$, but given the duality between top and socle, this choice is arbitrary.

Continuing to $\widetilde S_4$, some factors drop out because they are correspond to a block other than the two associate blocks of with label $(4)$, and we are left with $3!$ copies of each of two non-isomorphic bimodules of dimension twenty, each being the tensor product of the original left $k\widetilde S_7$-module and one of the two associate one-dimensional right $\widetilde S_4$-modules.

\end{examp}

We now construct the one-sided tilting complex, using the methods of \cite{RS}, \cite{SZ}.  By \cite{R1},\cite{RS}  there is a one-sided tilting complex consisting of linear complexes of projectives outward from the exceptional vertex. Letting each complex be labeled by its highest part, we would get:
\begin{alignat*}{8}
&  && && &&P_p\\
&  && &&P_{p+1} \rightarrow &&P_p\\
&  &&P_{p+2} \rightarrow &&P_{p+1} \rightarrow &&P_p\\
&  && && &&\dots\\
&  P_{p+s} \rightarrow &&\dots\rightarrow &&P_{p+1} \rightarrow &&P_p\\
&  && && &&P_{p-s-1}\rightarrow P_{p-s-2} \rightarrow &&\dots \rightarrow &&P_{p-t}\\
&  && && &&\dots\\
& && && &&P_{p-s-1}\rightarrow P_{p-s-2}\\
& && && &&P_{p-s-1}
\end{alignat*}
Letting each of these complexes be represented by $Q_r$, where $r$ is the highest index in the first set of complexes, and r is the lowest index in the second set of complexes, we find that one possible tilting complex which will lead to the complex obtained from the
representative for $s$ by removing the part $1$ is given by
\begin{alignat*}{9}
& && && && &&Q_{p}\\
& && && && &&Q_{p+1}\\
& && && && &&Q_{p+1}\rightarrow&&Q_{p+2}\\
& && && && && &&\dots\\
& && && && && &&Q_{p+s-3}\rightarrow&&Q_{p+s-2}\\
& && && && && && &&Q_{p+s-2}\rightarrow&&Q_{p+s-1}\\
& && && && && && && &&Q_{p+s-1}\rightarrow&&Q_{p+s}\\
&Q_{p-t}\rightarrow&&Q_{p-t+1}\\
& &&Q_{p-t+1}\rightarrow&&Q_{p-t+2}\\
& && && \dots\\
& && && &&Q_{p+s-1}\rightarrow&&Q_p\\
\end{alignat*}

To compose the two equivalences, we take mapping cones of the various combined complexes, leading to a new one-sided complex, with almost all terms in degree zero, except the term containing the part $p$. 
If we follow this procedure for each of the maps in the second tilting complex, taking care of the degrees of each complex, we get

\begin{alignat*}{8}
&  && &&P_{t+1}\\
&  && &&\dots\\
&  && &&P_{p-s-2}\\
& && &&P_{p-s-1}\rightarrow &&P_p\\
&  && && &&P_p\\
&  && &&P_{p+1} \rightarrow &&P_p\\
&  && &&P_{p+2}\\
&  && &&\dots\\
&  && &&P_{p+s}
\end{alignat*}

This complex is recognizable as an elementary one-sided tilting complex of the type  introduced by Okuyama \cite{O}.
 We have thus prepared to prove the following:
\begin{lem}  In a short string of weight type $(0,1,1,0)$, the derived equivalence between the two terms of weight $1$ is given by an elementary one-sided tilting complex.
\end{lem}
\begin{proof}
If $\rho^1$ and $\sigma^1$ are the middle terms, then they are both allowed equivalent to the middle terms of a $0$-string of the normal form given in the lemma, and thus there must be a tilting complex of the given form.

In order for the tilting complex to be elementary (in this case, to the right), we must have a subset $I_0$ of the set of indices $I$, such that the projective modules $P_i$ with indices in $I_0$ are all shifted one degree to the right, and every other projective module $P_j$ is replaced by a complex $P_j \rightarrow \bigoplus_{i \in I_0} P_i$, where the map $h$ is such that the kernel is the maximal submodule of $P_j$ which contains no composition factors $L_i$ with $i$ in $I_0$.  (The definition of elementary on the left is dual.) The only $P_j$ with composition factors labeled by indices in $I_0$ are the two adjacent projectives in the Brauer tree, which in this case are $P_{p+1}$ and $P_{p-s+1}$.  Since the maps in the tilting complex are maximal maps, the kernels have the desired maximality condition.

\end{proof}

Now that we have the divided powers, we can construct a candidate for the bimodule $\Theta$ which is to give the derived equivalence between the block of $A_5$ and the block of $A_6$.  Here we are in a case parallel to the case in which Rickard first proposed the complex of bimodules $\Theta$, in which the block on which we want $\Theta$ to act is at depth one in the string.  The complex $\theta$ is actually the direct sum of four complexes:
$$ \Theta = \Theta(3) \bigoplus \Theta(1) \bigoplus \Theta(-1) \bigoplus \Theta(-3).$$
 The numbers are the weights of the action of $h$ from the $sl_2$-action. The complexes at the extremal points are the divided powers corresponding to the source algebra equivalences derived in \cite{LS}.  The complexes in the middle are the tilting complex and its dual giving the derived equivalence.  Our candidate complex is
$$ E^{(1)} \rightarrow E^{(2)}\circ F^{(1)}.$$  This is actually dual to the complexes used in \cite{CR}, but this could have been fixed by reversing the numbering of the complexes that we used.

\section{The case of weight $2$}

The short strings of weight two are of two types, those with weight series (1,2,2,1), which are simply translations of strings with weight series (0,1,1,0), and those of weight series (0,2,2,0).  This difference is reflected clearly in the tilting complexes in the examples for which we have decomposition matrices, in that the examples of the types (1,2,2,1) which we calculated have tilting complexes which are derived directly from the tilting complex of the (0,1,1,0) example, whereas the tilting complex for the (0,2,2,0) example is an elementary complex with only one "active" module, that coming from the defect zero block.

      We checked all the examples for $p=5$ which fall into the range for which the decomposition matrices have been calculated, and checked that the transformations of the decomposition matrices were compatible with what we would expect from categorification, i.e., were indeed elementary.

      \begin{examp} The lowest rank  example of a string of type $(1,2,2,1)$ is from block $B = (\emptyset)^2$ of $\widetilde A_{10}$ to block $C = (1)^2$ of $\widetilde A_{11}$.  The block $B$  has five simple modules.   Of these, $L_{(5,4,1)}$ is connected in the crystal graph to  the simple module $L_{(5,4)}$ of $\widetilde A_9$, and similarly $L_{(4,3,2,1)}$ is connected in the crystal graph to  the simple module $L_{(4,3,2)}$.  After we have rearranged the rows to correspond by $K_0$  and the columns to correspond by the the crystal graph, we get that the decomposition matrix $D_C$ is obtained from the decomposition matrix $D_B$ of $B$ by

$$ D_B = \begin{matrix}
541&64&4321&73&532\\
[ 2& 1& 2& 0& 1 ]&~~~~(64)\\
[ 0& 1& 0& 1& 0 ]&~~~~(82)\\
[ 0& 1& 0& 1& 1 ]&~~~~(73)\\
[ 1& 1& 0& 0& 0 ]&~~~~(91)\\
[ 0& 0& 1& 0& 0 ]&~~(4321)\\
[ 1& 0& 0& 0& 0 ]&~~~~(10)\\
[ 1& 0& 0& 0& 0 ]&~~~~(10)\\
[ 0& 0& 1& 0& 1 ]&~~~(532)\\
[ 0& 0& 1& 0& 1 ]&~~~(532)\\
[ 1& 0& 1& 0& 1 ]&~~~(541)\\
[ 1& 0& 1& 0& 1 ]&~~~(541)
\end{matrix};$$

$$D_C = \begin{matrix}
551&641&5321&731&632\\
[ 2& 1& 2& 0& 1 ]&~~~~(641)\\
[ 0& 1& 0& 1& 0 ]&~~~~(821)\\
[ 0& 1& 0& 1& 1 ]&~~~~(731)\\
[ 1& 0& 0& 0& 0 ]&~~~~~(11)\\
[ 0& 0& 1& 0& 1 ]&~~~~(632)\\
[ 1& 1& 0& 0& 0 ]&~~~(10~1)\\
[ 1& 1& 0& 0& 0 ]&~~~(10~1)\\
[ 0& 0& 1& 0& 0 ]&~~~(5321)\\
[ 0& 0& 1& 0& 0 ]&~~~(5321)\\
[ 1& 1& 1& 0& 0 ]&~~~~~(65)\\
[ 1& 1& 1& 0& 0 ]&~~~~~(65)
\end{matrix};$$

By the method described in \cite{Sch},\cite{Al}, we calculated matrices $S$ and $M$ such that

$$S * D_C * M = D_G.$$

where  $S = diag[-1,1,1,-1,-1,-1,-1,-1,-1,-1,-1]$

$$
M = \left[\begin{array}{*{20}c}
   -1 & -1 & 0 & 0 & 0  \\
   0 & {  1} & 0 & 0 & 0  \\
   0 & 0 & -1 & 0 & -1  \\
   0 & 0 & 0 & {  1} & 0  \\
   0 & 0 & 0 & 0 & {  1}
 \end{array}\right]
$$

      The matrix $S$ records the parity in the shift of the irreducible modules over a field of characteristic $0$, while the matrix $M$ gives a "virtual" tilting complex. We number the projectives by the column in which they appear. In this case the matrix $M$ is compatible with a tilting complex:
      \begin{alignat*}{9}
      & && && && P_1\\
      & && && P_2 \rightarrow && P_1\\
      & && && && P_3\\
      & && && P_4\\
      & && && P_5 \rightarrow && P_3\\
    \end{alignat*}
    This is, as expected, a translation of the tilting complex for the string of type $(0,1,1,0)$.  Note that the column $P_4$, which is not derived from translation, is left unchanged.
    \end{examp}
    \begin{examp} The block of lowest rank for $p=5$ with weight type $(0,2,2,0)$ occurs in a $1$-string.  There is no such string for $p=3$ because there $t=1$, so there are no intermediate Scopes involutions, so this is indeed the example of lowest rank.

    In the $0$-strings of the types $(0,1,1,0)$ and $(1,2,2,1)$ there were no crossovers, because the parity of the simple modules remained the same for the entire string.  In $i$-strings for $i>0$ we have the parity switching back and forth.  Thus the two blocks of weight $2$ are the block $(1)^2$ in $\widetilde A_{1l}$, which we will denote as before by $C$, and the block $(2)^2$ in $\widetilde S_{12}$, which we will denote by $C'$.  There are both blocks of type $M$, and the decomposition matrices are as follows:

By the method described in \cite{Al}, we calculated matrices $S$ and $M$ such that

$$S * D_C * M = D_G.$$

where  $S = diag[-1,1,-1,1,-1,1,1,1,1,1,1]$

and

$$
M = \left[\begin{array}{*{20}c}
   1 & 0 & 0 & 0 & 0  \\
   0 & 1 & 0 & 0 & 0  \\
   0 & 0 & 1 & 0 & 0  \\
   0 & 0 & 0 & 1 & 0  \\
   -2 & -1 & -1 & -2 & {-1}
 \end{array}\right]
$$

      We are interested in the functor $E$ in the rank decreasing direction.

      $$\Theta(-1) =  E^{(2)}\circ F^{(1)} \rightarrow E^{(1)}.$$

      The elementary one-sided tilting complex that we deduce from the decomposition matrices has the form

          \begin{alignat*}{9}
      & && && P_5^2 \rightarrow &&P_1\\
      & && && P_5 \rightarrow &&P_2\\
      & && && P_5^2 \rightarrow &&P_3\\
      & && && P_5 \rightarrow &&P_4\\
      & && && P_5\\
    \end{alignat*}

In order to fill this out to a complex of bimodules of the form $\Theta(-1)$, we must take as many copies of each irreducible complex as the degree of the corresponding simple module in $B$.  In addition, since only copies of $P_5$ are expected in the term on the left, we presume that the terms which have been canceled in passing from the two-sided to the one-sided tilting complex are of the form
$$P_5 \rightarrow P_5,$$
with the identity map between the projectives.

\end{examp}

In summary, for blocks of type $M$, the permutation module method from \cite{KS} should allow the definition of divided powers.  For blocks of type $Q$, it appears that the permutation modules can produce divided powers, but the bimodules will usually not be indecomposable.


\begin{thebibliography}{HapS}

\bibitem[Al]{Al} O. Alon \textit{Proposed Tilting Complexes for Spin Blocks with defect group $C_5 \times C_5$}, Master's thesis, Bar-Ilan University, (2009).

\bibitem[AS]{AS} H. Arisha and M. Schaps \textit{Maximal strings in the crystal graph of spin representations of the symmetric and alternating groups}, Comm. in Algebra, \textbf{37}(2009), pp. 3779-3795.

\bibitem[B1]{B1} M. Brou\'e,   \textit{Isom\'etries parfaites, types de blocs, cat\'egories
d\'eriv\'ees}, Ast\'erisque  \textbf{181-182} (1990), 61--92.

\bibitem[B2]{B2} M. Brou\'e,   \textit{Rickard equivalences and block theory}, Groups
'93, Galway-Saint-Andrews Conference, vol. 1; London Math. Soc. Series, vol.
211, Cambridge University Press, 1995,  58--79.

\bibitem[BK1]{BK1} J. Brundan and A. Kleshchev, \textit{Projective representations of the symmetric groups via Sergeev duality}, Math. Z \textbf{239} (2002), 27-68.
\bibitem[BK2]{BK2} J. Brundan and A. Kleshchev, \textit{Blocks of cyclotomic Hecke algebras and Khovanov-Lauda algebras},arXiv:0808.2032, math.RT, math.QA
\bibitem[BK3]{BK3} J. Brundan and A. Kleshchev, \textit{Graded decomposition numbers for cyclotomic Hecke algebras },arXiv:0901.4450  math.RT, math.QA.
\bibitem[BKW]{BKW} J. Brundan, A. Kleshchev, and \textit{Graded Specht modules},arXiv:0901.0218 math.RT, math.QA.

\bibitem[Ca]{Ca} M. Cabanes, \textit{Local structure of the $p$-blocks of $\tilde S_n$}, Math. Z. \textbf{198} (1988), no.4, 519-543.



\bibitem [CK] {CK} J. Chuang and R. Kessar,  \textit{Symmetric groups, wreath products,
Morita equivalence, and Brou\'e's abelian conjecture}, Bull. London Math. Soc.,
to appear.

\bibitem [CR]{CR} J. Chuang and R. Rouquier, \textit{Derived equivalences for symmetric groups and $sl_2$ categorifications}, Annals of Mathematics (2) \textbf{167} no. 1(2008), 245-298..


\bibitem [K] {K}  R. Kessar, \textit{Blocks and source algebras for the double covers of the symmetric and alternating groups} J. Algebra {\bf 186} (1996), 872-933.

\bibitem [Kl] {Kl} A. Kleshchev, \textit{Linear and Projective Representations of Symmetric Groups},Cambridge Tracts in Mathematics, Cambridge Univ. Press(2005).
\bibitem [KS]{KS}   R. Kessar and M. Schaps, \textit{Crossover Morita equivalences for blocks
 of the covering groups of the symmetric and alternating groups}, J. Group Theory, \textbf{9} (2006),no. 6, 715-730.



\bibitem [LS]{LS} R. Leabovich and M. Schaps \textit{Crossover Morita equivalences of spin representations of the symmetric and alternating groups},  Bar-Ilan University arXiv: 0910.5070, math,RT.

\bibitem [LT]{LT} B. Leclerc and J.-Y. Thibon, \textit{$q$-deformed Fock spaces and modular representations of spin symmetric groups}, J. Physics. A \textbf{30} (1997), 6163-6176.

\bibitem [M]{M}J. Mueller \textit{Brauer trees for the Schur cover of the symmetric group} J. Algebra \textbf{266} (2003), no.2, 427-445.


\bibitem [O]{O} T. Okuyama,  \textit{Some examples of derived equivalent blocks of finite groups}, preprint,
Hokkaido, 1998.

\bibitem[R1] {R1} J. Rickard,  \textit{Morita theory for derived equivalence}, J.
London Math. Soc.
 \textbf{39} (1989), 436--456.


\bibitem [RS]{RS} J. Rickard, M. Schaps \textit{Folded tilting complexes for Brauer tree algebras} Advances in Mathematics,   .

\bibitem [Sc]{Sc} J. Scopes,
  \textit{Cartan matrices and Morita equivalence for blocks of the symmetric group},
J. Algebra\ textbf{142} (1991), 441--455.

\bibitem [Sch]{Sch} M. Schaps,  \textit{Deformations, tiltings and decomposition matrices}, Fields
Institute Communications   \textbf{45} (2005), 345-355.

\bibitem [St]{St} J. Stembridge, \textit{Shifted tableaux and the projective representations of symmetric groups.} Adv. Math. \textbf{74} (1989), no. 1, 87--134.

\bibitem [SZ] {SZ} M. Schaps, E. Zakay, \textit{Pointed Brauer trees}, Journal of Algebra \textbf{246} (2001), 647-672.

\bibitem[Wh]{Wh} M. White \textit{One-Sided Tilting Complexes for Blocks of Defect $1$ in the Covering Groups of the Symmetric and Alternating Groups}, Master's thesis, Bar-Ilan University, (2009).

\end{thebibliography}
\end{document}